\documentclass[11pt]{article}
\usepackage{amssymb,amsmath,amsfonts,amsthm,color,esint,graphicx}
\usepackage{epsfig}

\numberwithin{equation}{section}

\newcommand{\cred}{\color{red}}

\newcommand{\di}{\displaystyle}

\newcommand{\R}{{\mathbb R}}
\newcommand{\RN}{{\mathbb R}^N}

\def\bu{\bar{u}}

\newtheorem{theorem}{Theorem}[section]

\newtheorem{corollary}[theorem]{Corollary}

\newtheorem{lemma}[theorem]{Lemma}

\newtheorem{proposition}[theorem]{Proposition}

\begin{document}

\title{\vskip-0.3in Polyharmonic inequalities with nonlocal terms}

\author{Marius Ghergu\footnote{School of Mathematics and Statistics,
    University College Dublin, Belfield, Dublin 4, Ireland; {\tt
      marius.ghergu@ucd.ie}}\;\,\footnote{Institute of Mathematics Simion Stoilow of the Romanian Academy, 21 Calea Grivitei St., 010702 Bucharest, Romania}
      $\;\;$        $\;$
{Yasuhito Miyamoto \footnote{Graduate School of Mathematical Sciences, The University of Tokyo, 3-8-1 Komaba, Meguro-ku, Tokyo 153-8914, Japan; {\tt  miyamoto@ms.u-tokyo.ac.jp}}\; \footnote{Supported by JSPS KAKENHI Grant Numbers 
19H01797, 19H05599.}}
 $\;\;$    and    $\;$
{Vitaly Moroz\footnote{Mathematics Department, Swansea University,
Bay Campus, Fabian Way, Swansea SA1 8EN, Wales, United Kingdom; {\tt
        v.moroz@swansea.ac.uk}}}
}
\maketitle

\begin{abstract}
We study the existence and non-existence of classical solutions for inequalities of type
$$
\pm \Delta^m u \geq \big(\Psi(|x|)*u^p\big)u^q \quad\mbox{ in }\R^N (N\geq 1).
$$
Here, $\Delta^m$ $(m\geq 1)$ is the polyharmonic operator,  $p, q>0$ and $*$ denotes the convolution operator, where $\Psi>0$ is a continuous non-increasing function. We devise new methods to deduce that solutions of the above inequalities satisfy the poly-superharmonic property.  This further allows us to 
obtain various Liouville type results. Our study is also extended to the case of systems of simultaneous inequalities. 
\end{abstract}

\bigskip


\medskip


\bigskip

\section{Introduction and the main results}

In this paper we are concerned with the following  elliptic inequality
\begin{equation}\label{gen}
\pm \Delta^m u \geq \big(\Psi(|x|)*u^p\big)u^q \quad\mbox{ in }\R^N, 
\end{equation}
and its corresponding systems of inequalities, where $N,m\geq 1$ are integers,  $p, q>0$ and $\Delta^m$ denotes the $m-$polyharmonic operator. The function $\Psi$ satisfies

\begin{equation}\label{eqK}
\begin{cases}
\Psi>0 \quad \mbox{ and } \quad \Psi(|x|)\in C(\R^N\setminus\{0\})\cap L^1_{loc}(\R^N);\\[0.1in]
\Psi(r) \mbox{ is non-increasing on } (0, \infty);\\[0.1in]
 \displaystyle \lim_{r\to\infty}r^N\Psi(r)=\infty.
\end{cases}
\end{equation}
Typical examples of functions $\Psi$ include:
$$
\Psi(r)=r^{-\alpha}, \;\;0<\alpha<N \;\;\;\mbox{ or } \;\;\; 
\Psi(r)=r^{-N}\log^{-\beta}\Big(1+\frac{1}{r}\Big), \;\;
1<\beta\le N.
$$
The operator $*$ in \eqref{gen} is the standard convolution operator, that is, 
$$
\Psi(|x|)*u^p=\int_{\R^N}u^p(y)\Psi(|y-x|)dy.
$$ 
By a non-negative solution $u$ of \eqref{gen} we understand a function $u\in C^{2m}(\R^N)$, $u\geq 0$, such that
\begin{equation}\label{eqkk}
\int_{|y|>1} u^p(y)\Psi\left(\frac{|y|}{2}\right) dy<\infty
\end{equation}
and $u$ satisfies \eqref{gen} pointwise.

Note that condition \eqref{eqkk} is weaker than the condition $u\in L^p(\R^N)$ and that \eqref{eqkk} is (almost) necessary and sufficient in order to ensure that the convolution term $\Psi(|x|)*u^p$ is finite for all $x\in \R^N$. Indeed, for any $x\in \R^N$ we have
$$
\begin{aligned}
\Psi(|x|)*u^p&=\int_{|y|\leq 2|x|}u^p(y)\Psi(|y-x|)dy+\int_{|y|>2|x|}u^p(y)\Psi(|y-x|)dy\\
&\leq \Big( \max_{|z|\leq 2|x|}u(z) \Big)^p \int_{B_{3|x|}}\Psi(|y|)dy+\int_{|y|>2|x|} u^p(y)\Psi\left(\frac{|y|}{2}\right) dy<\infty,
\end{aligned}
$$
by the fact that $\Psi(|x|)\in L^1_{loc}(\R)$ is non-increasing and \eqref{eqkk}.

We prefer to separate the analysis of \eqref{gen} into two distinct inequalities as follows:
\begin{equation}\label{main0}
-(-\Delta)^m u \geq \big(\Psi(|x|)*u^p\big)u^q \quad\mbox{ in }\R^N, 
\end{equation}
and
\begin{equation}\label{main}
(-\Delta)^m u \geq \big(\Psi(|x|)*u^p\big)u^q \quad\mbox{ in }\R^N.
\end{equation}

We study first the inequality \eqref{main0}. Our main result in this case reads as follows.

\begin{theorem}\label{th0}
Assume 
\begin{equation}\label{c1}
N,m\geq 1 \quad\mbox{ and }\quad p+q\geq 2,
\end{equation}
or
\begin{equation}\label{c2}
N>2m \quad\mbox{ and }\quad p\geq 1.
\end{equation}
If $u\in C^{2m}(\RN)$ is a non-negative solution of \eqref{main0}, then $u\equiv 0$. 
\end{theorem}

The above result is new even in the semilinear case $m=1$. 

We are next concerned with the inequality \eqref{main}. 
The related semilinear problem
\begin{equation}\label{semilin}
-\Delta u \geq \big(|x|^{-\alpha}*u^p\big)u^q \quad\mbox{ in }\R^N
\end{equation}
was completely investigated in \cite{MV2013}. The equality case in \eqref{semilin} bears the name Choquard-Pekar (or simply Choquard) equation and was introduced in \cite{P1954}  as a model in quantum theory. Since then, the prototype model \eqref{semilin} has been used to describe many phenomena arising in mathematical physics  (see, e.g., \cite{J1995, MRP1998,MV2017} for further details).
Quasilinear versions of \eqref{semilin} (including the case of $p$-Laplace or mean curvature operator) are discussed in \cite{CMP2008,GKS2020a, GKS2020b}. Also, singular solutions of \eqref{semilin} are considered in \cite{CZ2016, CZ2018, FG2020, GT2015, GT2016}.  It is obtained in \cite[Theorem 1]{MV2013} (see also \cite[Corollary 2.5]{GKS2020a} for an extension to quasilinear inequalities)
that if $p\geq 1$ and $q>1$ then \eqref{semilin} has non-negative solutions if and only if
\begin{equation}\label{condz}
\min\{p,q\} >\frac{N-\alpha}{N-2}\quad\mbox{ and }\quad p+q>\frac{2N-\alpha}{N-2}.
\end{equation}

An important matter in the study of polyharmonic problems is whether non-negative solutions of $(-\Delta)^mu \geq f(u)$ in $\R^N$ enjoy the so-called {\it poly-superharmonic} property, that is, whether $(-\Delta)^j u\geq 0$ in $\R^N$ for all $1\leq j\leq m$. This has been shown to be true under some general conditions for nonlinearities $f(u)$ without non-local terms (see, e.g. \cite{CL2013, GL2008, L2013, LGZ2006, N2017}). In \cite{GL2008, LGZ2006} the authors use a contradiction argument and construct a suitable sequence of power functions acting as lower barriers for the spherical average of the solution $u$. In turn, the approach in \cite{CL2013, L2013} relies on a re-centering argument which at each step in the proof one brings forward a new center of the spherical average of the functions $(-\Delta)^j u$.

 In this paper we show that the poly-superharmonic property is still preserved in case of polyharmonic inequalities of type \eqref{main} in the presence of convolution terms. More precisely we have:

\begin{theorem}\label{thmain1}
Assume $N,m\geq 1$ and either $p+q\geq 2$ or $p\geq 1$. 

If $u\in C^{2m}(\R^N)$ is a non-negative solution of \eqref{main}, then, for all $1\leq j\leq m$ we have
$$
(-\Delta)^j u\geq 0\quad\mbox{ in }\R^N.
$$
\end{theorem}

Our approach is different from the methods already devised in \cite{CL2013, L2013, LGZ2006, N2017}. More precisely,  if $p+q\geq 2$ we use a general integral estimate which we establish in Lemma \ref{lg} below. On the other hand, if $p\geq 1$, we exploit the integral condition \eqref{eqkk} to construct our argument.

Particularly useful are the integral representation formulae obtained in \cite{CDM2008}  (see also \cite{DM2014, DMP2006}) for solutions of $(-\Delta)^m =\mu$ in ${\mathcal D}'(\R^N)$, where $\mu$ is a Radon measure on $\R^N$, $N>2m$. Such an approach was recently adopted in \cite{NY2020} for the polyharmonic Hardy-H\'enon equation $(-\Delta)^m u=|x|^\sigma u^p$ in $\R^N$, $p>1$.
We believe that this is still an underexploited direction of research in the study of higher order elliptic equations and inequalities. We shall make use of these facts which we recall in Proposition \ref{prep}.

We next focus on Liouville-type results for the inequality \eqref{main}.

\begin{theorem}\label{thmain2}

Assume $N,m\geq 1$ and let $u\in C^{2m}(\RN)$ be a non-negative solution of \eqref{main}.

\begin{enumerate}
\item[\rm (i)] If $1\leq N\leq 2$  and ($p\geq 1$ or $p+q\geq 2$), then $u\equiv 0$.  
\item[\rm (ii)] If $N>2m$ and one of the following conditions  holds:
\begin{enumerate}
\item[\rm (ii1)] ($p\geq 1$ or $p+q\geq 2$) and $\di \int_{|y|>1}|y|^{-p(N-2m)}\Psi(|y|)dy=\infty$; 
\item[\rm (ii2)]  $p+q\geq 2$ and $\di \limsup_{r\to \infty} r^{2N-(N-2m)(p+q)}\Psi(r)>0$; 
\end{enumerate}
then,  $u\equiv 0$.  
\end{enumerate}
\end{theorem}

In the case $\Psi(r)=r^{-\alpha}$, $\alpha\in (0, N)$, we obtain the following result. 
\begin{theorem}\label{thmain3}
Assume $N>2m$, $m\geq 1$, $\alpha\in (0,N)$,  $p\geq 1$ and $q>1$. Then, the inequality
\begin{equation}\label{eqpqm1}
(-\Delta)^m u \geq \big(|x|^{-\alpha}*u^p\big)u^q \quad\mbox{ in }\R^N
\end{equation}
has non-negative non-trivial solutions  if and only if 
\begin{equation}\label{eqpqm2}
\min\{p,q\}>\frac{N-\alpha}{N-2m}\quad\mbox{ and }\quad p+q>\frac{2N-\alpha}{N-2m}.
\end{equation}
\end{theorem}

In the particular case $m=1$, Theorem \ref{thmain3} retrieves the result in \cite[Theorem 1]{MV2013} which yields the optimal condition \eqref{condz} for the existence of non-negative solutions of the semilinear inequality \eqref{semilin}.   

\smallskip

We next discuss the corresponding systems associated to \eqref{main}, namely

\begin{equation}\label{mainsystem}
(-\Delta)^mu_i\ge \sum_{j=1}^ne_{ij} \left( \Psi_{ij}(|x|)*u_j^{p_{ij}}\right)u_j^{q_{ij}}\ \ \textrm{in}\ \ \RN, \; 1\leq i\leq n 
\end{equation}
and
\begin{equation}\label{AT1E0}
(-\Delta)^mu_i\ge\sum_{j=1}^ne_{ij}\left( \Psi_{ij}(|x|)*u_j^{p_{ij}}\right) u_i^{q_{ij}}
\ \ \textrm{in}\ \ \RN, \; 1\leq i\leq n 
\end{equation}
where $N,m\geq 1$, $p_{ij}\ge 1$, $q_{ij}>0$ and $(e_{ij})$ is the adjacency matrix, {\it i.e.}, $e_{ij}$ satisfies
\[
e_{ij}=0\ \textrm{or}\ 1\quad\textrm{and}\quad e_{ij}=e_{ji}\quad \textrm{for}\ i,j\in\{1,\ldots,n\}.
\]

By a non-negative solution of \eqref{mainsystem} (resp. \eqref{AT1E0}) we understand a $n$-component function $u=(u_1, u_2, \dots, u_n)$ with $u_j\in C^{2m}(\R^N)$, $u_j\geq 0$, such that
\begin{equation}\label{eqkk01}
\sum_{i=1}^n \int_{|y|>1} u_j^{p_{ij}}(y)\Psi_{ij}\left(\frac{|y|}{2}\right) dy<\infty
\end{equation}
and $u$ satisfies \eqref{mainsystem} (resp. \eqref{AT1E0}) pointwise.

The main result for (\ref{mainsystem}) is the following:
\begin{theorem}\label{S4T2}
Assume $N>2m$, $m\ge 1$, and let $(u_1,\ldots,u_n)$ be a non-negative solution of (\ref{mainsystem}).
If there exist $k,\ell\in\{1,\ldots,n\}$ (not necessarily distinct) such that 
$$
e_{k\ell}=e_{\ell k}=1,
$$
\begin{equation}\label{pkl}
p_{k\ell}+q_{k\ell}\geq 2,\quad  p_{\ell k}+q_{\ell k}\geq 2,
\end{equation}
\begin{equation}\label{S4T2E0}
\limsup_{r\to\infty} \min\Big\{ r^{2N-(N-2m)(p_{k\ell}+q_{k\ell})}\Psi_{k\ell}(r)\,,\, 
r^{2N-(N-2m)(p_{\ell k}+q_{\ell k})}\Psi_{\ell k}(r)\Big\}>0,
\end{equation}
then $u_k\equiv u_\ell\equiv 0$.

\end{theorem}

Theorem~\ref{S4T2} states that if $k$ is adjacent to $\ell$ and \eqref{pkl}-\eqref{S4T2E0} hold, then $u_k\equiv u_{\ell}\equiv 0$.
We immediately obtain the following Liouville-type result:
\begin{corollary}\label{corol1}
Make the same assumptions as in Theorem~\ref{S4T2}.
In particular, suppose that (\ref{pkl}) and (\ref{S4T2E0}) hold for every pair $(k,\ell)$ such that $e_{k\ell}=e_{\ell k}=1$.
If each connected component of the network has more than two nodes, then 
the only non-negative solution of (\ref{mainsystem}) is $(0,\ldots,0)$.
In particular, if the network has only one connected component and $n\ge 2$, then the only non-negative solution of (\ref{mainsystem}) is $(0,\ldots,0)$.
\end{corollary}

The main result regarding the system \eqref{AT1E0} is the following.

\begin{theorem}\label{S4T3}
Assume $N>2m$, $m\ge 1$, and let $(u_1,\ldots,u_n)$ be a non-negative solution of (\ref{AT1E0}).
If there exist $k,\ell\in\{1,\ldots,n\}$ (not necessarily distinct) such that 
$$
e_{k\ell}=e_{\ell k}=1,
$$
\begin{equation}\label{pkl1}
p_{k\ell}, q_{k\ell}\geq 1,\quad  p_{\ell k},q_{\ell k}\geq 1,
\end{equation}
and (\ref{S4T2E0}) holds, then $u_k\equiv 0$ or $u_\ell\equiv 0$ (or both).
\end{theorem}

Assume next that the adjacency matrix $(e_{ij})$ is given by
\begin{equation}\label{e}
e_{ij}=
\begin{cases}
1 & \textrm{if}\ i\neq j,\\
0 & \textrm{if}\ i=j.
\end{cases}
\end{equation}
From Theorem \ref{S4T2} and Theorem \ref{S4T3} we find:
\begin{corollary}\label{corol4}
Suppose $N>2m$, $m\geq 1$ and that $(e_{ij})$ is defined by (\ref{e}).
\begin{enumerate}
\item[\rm (i)] Assume (\ref{pkl})-(\ref{S4T2E0}). Then the only non-negative solution of (\ref{mainsystem}) is
\[
(u_1,\ldots,u_n)=(0,\ldots,0).
\]
\item[\rm (ii)] Assume (\ref{S4T2E0})-(\ref{pkl1}). Then all non-negative solutions of (\ref{AT1E0}) are of the form 
\[
(u_1,\ldots,u_n)=(0,\ldots,0,u_j,0,\ldots,0)
\ \ \textrm{for some}\ j\in\{1,\ldots,n\},
\]
where $(-\Delta)^mu_j\ge 0$ in $\R^N$.
\end{enumerate}
\end{corollary}


\section{Preliminaries}

In this section we collect some auxiliary results which will be useful in our proofs.

\begin{lemma}\label{lbas}
Let $\alpha\in (0,N)$, $\beta>N-\alpha$ and $f\in L^1_{loc}(\R^N)$, $f\geq 0$, be such that 
$$
f(x)\leq c|x|^{-\beta}\quad\mbox{ for all } x\in \R^N\setminus B_\rho,
$$ 
where $c>0$ and $\rho>1/2$.  Then, there exists a constant $C=C(N,\alpha, \beta, \rho)>0$ such that for all $x\in \R^N\setminus B_{2\rho}$ one has 
$$
\int_{\R^N} \frac{f(y)}{|x-y|^{\alpha}} dy
\leq
C\left\{
\begin{aligned}
&|x|^{N-\alpha-\beta}&&\quad\mbox{ if }\beta<N,\\
&|x|^{-\alpha}\log |x|&&\quad\mbox{ if }\beta=N,\\
&|x|^{-\alpha} &&\quad\mbox{ if }\beta>N.
\end{aligned}
\right.
$$
\end{lemma}
Similar estimates are available in \cite[Lemma 2.1]{GKS2020b}, \cite[Lemma A.1]{MV2013} (see also \cite{GT2015,GT2016book}).

\begin{proof} We have 
\begin{equation*}
\int_{\R^N} \frac{f(y)}{|x-y|^{\alpha}} dy  = \int_{|y|\geq 2|x|} \frac{f(y)}{|x-y|^{\alpha}} dy +
\int_{\frac12 |x| \leq |y| \leq 2|x|} \frac{f(y)}{|x-y|^{\alpha}} dy +
\int_{|y|\leq  |x|/2} \frac{f(y)}{|x-y|^{\alpha}} dy.
\end{equation*}
For $|y|\geq 2|x|$ we have $|x-y| \geq |y|-|x| \geq |y|/2$, so that
\begin{equation*}
\int_{|y|\geq 2|x|} \frac{f(y)}{|x-y|^{\alpha}} dy
\leq C \int_{|y|\geq 2|x|} \frac{dy}{|y|^{\alpha+\beta}}
\leq C|x|^{N-\alpha -\beta}.
\end{equation*}
Similarly we estimate
\begin{align*}
\int_{\frac12 |x| \leq |y| \leq 2|x|} \frac{f(y)}{|x-y|^{\alpha}} dy
&\leq C|x|^{-\beta}  \int_{\frac12 |x| \leq |y| \leq 2|x|}
\frac{dy}{|x-y|^{\alpha}}\\
&\leq C|x|^{-\beta}  \int_{ |y-x| \leq 3|x|}
\frac{dy}{|x-y|^{\alpha}}=C|x|^{N-\alpha-\beta}.
\end{align*}
Finally, if $|y|\leq |x|/2$ then $|x-y| \geq |x|-|y| \geq |x|/2$. Hence
\begin{align*}
\int_{|y|\leq  |x|/2} \frac{f(y)}{|x-y|^{\alpha}} dy &\leq
C|x|^{-\alpha}  \int_{|y|\leq  |x|/2} f(y) dy\\
&\leq
C|x|^{-\alpha}  \left\{ \int_{|y|\leq  \rho} f(y) dy+\int_{\rho<|y|\leq  |x|/2} f(y) dy\right\}\\
&\leq
C|x|^{-\alpha}  \left\{ 1+C\int_{\rho<|y|\leq  |x|/2} |y|^{-\beta} dy\right\}\\
&\leq
C|x|^{-\alpha} +C
\begin{cases}
|x|^{N-\alpha-\beta}&\quad\mbox{ if }\beta<N,\\
|x|^{-\alpha}\log |x|&\quad\mbox{ if }\beta=N,\\
|x|^{-\alpha} &\quad\mbox{ if }\beta>N.
\end{cases}
\end{align*}
The result now follows by combining the above three estimates.
\end{proof}

\begin{lemma}\label{ln} {\rm (see \cite[Lemma 3.4]{NNPY2020})}
Let $u\in C^{2m}(\R^N)$ be such that $u\geq 0$ and $(-\Delta)^m u\leq 0$ in $\R^N$. If
$$
\int_{B_R}u dx=o(R^N)\quad\mbox{ as }R\to \infty,
$$
then $u\equiv 0$.
\end{lemma}

A crucial result in our approach is the following representation formula for distributional solutions of the polyharmonic operator.

\begin{proposition}\label{prep} {\rm (see \cite[Theorem 2.4]{CDM2008})}
Let $m\geq 1$ be an integer and $N>2m$. Suppose $\mu$ is a positive Radon measure on $\R^N$ and $\ell\in \R$. 
The following statements are equivalent:

\begin{enumerate}
\item[\rm (i)] $u\in L^1_{loc}(\R^N)$ is a distributional solution of 
\begin{equation}\label{distr}
(-\Delta)^m u=\mu\quad\mbox{ in } \mathcal{D}'(\R^N),
\end{equation} 
and for a.e. $x\in \R^N$ we have
\begin{equation}\label{rep0}
\liminf_{R\to \infty}\frac{1}{R^N}\int\limits_{R\leq |y-x|\leq 2R}|u(y)-\ell|dy=0.
\end{equation}

\item[\rm (ii)] $u\in L^1_{loc}(\R^N)$ is a distributional solution of \eqref{distr}, ${\rm essinf}\,u=\ell$ and $u$ is weakly polysuperharmonic in the sense that
$$
\int_{\R^N} u(-\Delta)^i\varphi \geq 0\quad\mbox{ for all }1\leq i\leq m, \varphi\in C^\infty_0(\R^N), \varphi\geq 0.
$$

\item[\rm (iii)] $u\in L^1_{loc}(\R^N)$ and there exists $c=c(N,m)>0$ such that
$$
u(x)=\ell+c\int_{\R^N}\frac{d\mu(y)}{|x-y|^{N-2m}}\quad\mbox{ for a.e. }x\in \R^N.
$$
\end{enumerate}
\end{proposition}

Using Proposition \ref{prep} we deduce:

\begin{lemma}\label{lemv}
Let $v\in L^1_{loc}(\R^N)$ be a distributional solution of 
$$
(-\Delta)^m v=f \quad\mbox{ in } \mathcal{D}'(\R^N)\, , \,  N>2m,
$$
where $ f\in L^1_{loc}(\R^N)$, $f\geq 0$, $f\not\equiv 0$. Assume that 
\begin{equation}\label{eqkk1}
\int_{|y|>1} |v|^p(y)\Psi\left(\frac{|y|}{2}\right) dy<\infty,
\end{equation}
where $p\geq 1$ and $\Psi$ is a function which satisfies \eqref{eqK}. Then, ${\rm essinf}\, v=0$ and for some constant $c>0$ we have 
\begin{equation}\label{repv}
v(x)=c\int_{\R^N}\frac{f(y)}{|x-y|^{N-2m}}dy \quad\mbox{ for a.e. }x\in \R^N.
\end{equation}
In particular $v> 0$ in $\R^N$ and 
\begin{equation}\label{inqq9}
v(x)\geq c|x|^{2m-N}\quad\mbox{ in }\R^N\setminus B_1,
\end{equation}
for some constant $c>0$.
\end{lemma}
\begin{proof} We show that $v$ satisfies condition \eqref{rep0} in Proposition \ref{prep} with $\ell=0$. First, if $p>1$ by H\"older's inequality and \eqref{eqkk1}, for all $x\in \R^N$  we have
\begin{equation}\label{cl1}
\begin{aligned}
\int\limits_{R\leq |y-x|\leq 2R} |v(y)|dy & \leq \Big( \int\limits_{R\leq |y-x|\leq 2R} |v|^p(y) \Psi \Big(\frac{|y|}{2}\Big)  dy\Big)^{\frac{1}{p}}
\Big( \int\limits_{R\leq |y-x|\leq 2R} \Psi^{-\frac{1}{p-1}}\Big(\frac{|y|}{2}\Big)  dy\Big)^{1-\frac{1}{p}}\\
&\leq C \Big( \int\limits_{R\leq |y-x|\leq 2R} \Psi^{-\frac{1}{p-1}}\Big(\frac{|y|}{2}\Big) dy\Big)^{1-\frac{1}{p}}\\
&= C \Big( \int\limits_{R\leq |z|\leq 2R} \Psi^{-\frac{1}{p-1}}\Big(\frac{|z+x|}{2}\Big) dz\Big)^{1-\frac{1}{p}}.
\end{aligned}
\end{equation}
Take $R>|x|$. By the property \eqref{eqK} on $\Psi$ we have
$$
\Psi^{-\frac{1}{p-1}}\Big(\frac{|z+x|}{2}\Big) \leq \Psi^{-\frac{1}{p-1}}\Big(\frac{3R}{2}\Big) =o(R^{\frac{N}{p-1}})\quad\mbox{ for all }R<|z|<2R
$$
as $R\to\infty$. Thus, we may further estimate in \eqref{cl1} to deduce
\begin{equation}\label{cl2}
\frac{1}{R^N} \int\limits_{R\leq |y-x|\leq 2R} |v(y)|dy
=o(1)\to 0\quad\mbox{ as }R\to\infty.
\end{equation}
If $p=1$ we simply use the property \eqref{eqK} and \eqref{eqkk} to estimate
$$
\begin{aligned}
\int\limits_{R\leq |y-x|\leq 2R} |v(y)|dy &\leq \int\limits_{R\leq |y-x|\leq 2R}|v(y)| \Psi \Big(\frac{|y|}{2}\Big)  \Psi^{-1} \Big(\frac{|y|}{2}\Big)dy\\
&\leq  \Psi^{-1} \Big(\frac{3R}{2}\Big)   \int\limits_{R\leq |y-x|\leq 2R}|v(y)| \Psi \Big(\frac{|y|}{2}\Big) dy\\
&\leq  C\Psi^{-1} \Big(\frac{3R}{2}\Big)  
=o(R^N)
\end{aligned}
$$
for all $R>|x|$  as $R\to\infty$. This shows that \eqref{cl2} also holds for $p=1$.
Hence, $v$ satisfies the condition \eqref{rep0} in Proposition \ref{prep} with $\ell =0$. The representation integral \eqref{repv} follows by Proposition \ref{prep}  while the estimate \eqref{inqq9} follows from \cite[Lemma 3.1]{CDM2008}.
\end{proof}

\section{Proof of Theorem \ref{th0}}

We first establish an estimate which holds for the general inequality \eqref{gen}.

\begin{lemma}\label{lg}
Let $\psi\in C_c^\infty(\R^N)$ be such that supp $\psi\subset B_{2}$, $0\leq \psi\leq 1$, $\psi\equiv 1$ on $B_1$. 
For $R>2$ define 
$\varphi(x)=\psi^{2m} (x/R)$.
Then, there exists $C>0$ such that any non-negative solution $u$ of \eqref{gen} satisfies
\begin{equation}\label{eqg1}
\int_{\R^N} u\varphi dx\geq CR^{-N+2m}\Big(\int_{\R^N} u^{\frac{p+q}{2}}\varphi   dx\Big)^2.
\end{equation}

\end{lemma}
\begin{proof} 
It is easy to check that
$$
\left|\Delta^m (\psi^{4m})\right| \leq C\psi^{2m} \quad \mbox{ in }\R^N,
$$
where $C>0$ is a positive constant. This yields
$$
{\left|\Delta^m (\varphi^2)\right|\leq \frac{C}{R^{2m}} \varphi \quad \mbox{ in }\R^N}.
$$
We multiply by $\varphi^2$ in \eqref{gen} and integrate. Using the above estimate we find
\begin{equation}\label{eqg2}
\begin{aligned}
\int_{\R^N} \big(\Psi(|x|)*u^p\big)u^q\varphi^2 &\leq \pm \int_{\R^N} \varphi^2 (\Delta^m u)
=\pm \int_{\R^N} u \Delta^m (\varphi^2) \\
&{\le \int_{\R^N}u\left|\Delta^m(\varphi^2)\right|} 
\leq \frac{C}{R^{2m}}\int_{B_{2R}} u\varphi.
\end{aligned}
\end{equation}
We next estimate the left-hand side of \eqref{eqg2}. 
By inter-changing the variables and H\"older's inequality one gets 
$$
\begin{aligned}
\Big( \int_{\R^N} &\big(\Psi(|x|)*u^p\big)u^q \varphi^2 dx \Big)^2\\
=\,&\Big(\iint_{\R^N\times \R^N} \Psi(|x-y|)u^p(x) u^q(y) \varphi^2(y) dxdy\Big)^2\\
=\,&\Big(\iint_{\R^N\times \R^N} \Psi(|x-y|)u^p(x) u^q(y) \varphi^2(y) dxdy\Big) \Big(\iint_{\R^N\times \R^N} \Psi(|x-y|)u^p(y) u^q(x) \varphi^2(x) dxdy\Big)\\
\geq\, & \Big(\iint_{\R^N\times \R^N} \Psi(|x-y|)u^{\frac{p+q}{2}}(x) u^{\frac{p+q}{2}}(y) \varphi(x) \varphi(y)dxdy\Big)^2\\
\geq \,& \Psi(4R)^{2} \Big(\iint_{B_{2R}\times B_{2R}} u^{\frac{p+q}{2}}(x) u^{\frac{p+q}{2}}(y) \varphi(x) \varphi(y)dxdy\Big)^2 \\
\geq \,& cR^{-2N} \Big(\iint_{B_{2R}\times B_{2R}} u^{\frac{p+q}{2}}(x) u^{\frac{p+q}{2}}(y) \varphi(x) \varphi(y)dxdy\Big)^2,
\end{aligned}
$$
where $c>0$ is a constant.
Now, splitting the integrals according to $x$ and $y$ variables we deduce
$$
\int_{\R^N} \big(\Psi(|x|)*u^p\big)u^q dx \geq cR^{-N}
\Big(\int_{\R^N} u^{\frac{p+q}{2}}(x)  \varphi(x) dx  \Big)^2.
$$
Using this last inequality in \eqref{eqg2} we deduce \eqref{eqg1}.
\end{proof}
\medskip

\noindent{\bf Proof of Theorem \ref{th0} completed.}
We shall discuss separately the cases where \eqref{c1} or \eqref{c2} holds.

\noindent{\bf Case 1:} $N,m\geq 1$ and $p+q\geq 2$.
Let $\varphi$ be as in Lemma \ref{lg}.  By H\"older's inequality we have
$$
\begin{aligned}
\int_{\R^N} u\varphi&\leq 
\Big(  \int_{\R^N}u^{\frac{p+q}{2}}\varphi\Big)^{\frac{2}{p+q}}
\Big(  \int_{\R^N}\varphi \Big)^{1-\frac{2}{p+q}}\\
&\leq C R^{N(1-\frac{2}{p+q})} \Big(  \int_{\R^N} u^{\frac{p+q}{2}}\varphi\Big)^{\frac{2}{p+q}}.
\end{aligned}
$$
Hence,
$$
\Big(  \int_{\R^N} u^{\frac{p+q}{2}}\varphi \Big)^{2}\geq C R^{-N(p+q-2)} \Big(\int_{\R^N} u\varphi\Big)^{p+q}.
$$
Using this last estimate in \eqref{eqg1} we find 
$$
\int_{\R^N} u\varphi \leq CR^{N-\frac{2m}{p+q-1}}\quad\mbox{ for all }R>2.
$$
In particular, since $\varphi=1$ on $B_{R}$ we deduce
$$
\int_{B_R} u  dx =o(R^N)\quad\mbox{ as }R\to \infty. 
$$
By Lemma \ref{ln} it now follows that $u\equiv 0$ which concludes our proof in this case.

\medskip

\noindent{\bf Case 2:} $N>2m$ and $p\geq 1$.We apply Lemma \ref{lemv} for $v=-u$. It follows in particular that $v=-u\geq 0$ which yields $u\equiv 0$.
\qed

\section{Proof of Theorem \ref{thmain1}}

The proof of Theorem \ref{thmain1} follows from the result below which will also be useful in the study of the system \eqref{mainsystem}.
\begin{theorem}\label{thmainext}
 Assume $N,m\geq 1$ and let $\Phi$ and $\Psi$  satisfy \eqref{eqK}. Suppose  that $(u, v)$ is a non-negative solution of 
\begin{equation}\label{nlc}
\begin{cases}
(-\Delta)^m u \geq \big(\Phi(|x|)*v^{p_1}\big)v^{q_1}  \\[0.1in]
(-\Delta)^m v \geq \big(\Psi(|x|)*u^{p_2}\big)u^{q_2} 
\end{cases}
\quad\mbox{ in }\R^N, 
\end{equation}
where either
\begin{equation}\label{ss1}
p_1+q_1\geq 2\quad\mbox{ and }\quad p_2+q_2\geq 2,
\end{equation}
or
\begin{equation}\label{ss2}
p_1,p_2\geq 1.
\end{equation}

Then, for all $1\leq i\leq m$ we have 
$$
(-\Delta)^i u\geq 0 \quad \mbox{ and }\quad (-\Delta)^i v\geq 0\quad \mbox{ in }\R^N.
$$
\end{theorem}
\begin{proof} If $u\equiv 0$ (resp. $v\equiv 0$), then $v\equiv 0$ (resp. $u\equiv 0$), otherwise
\[
0\ge\int\int \Phi(|x-y|)v^{p_1}(y)v^{q_1}(y)dxdy>0
\]
which is a contradiction.
Hereafter we assume that $u\not\equiv 0$ and $v\not\equiv 0$.
The proof is divided in two steps. 
\medskip

\noindent{\bf Step 1: }
We have $(-\Delta)^{m-1}u\geq 0$ and $(-\Delta)^{m-1}v\geq 0$ in $\R^N$.

Assume by contradiction that there exists $x_0\in \R^N$ such that $(-\Delta)^{m-1}u(x_0)<0$. Let $u_i=(-\Delta)^i u$ and $v_i=(-\Delta)^i v$, $1\leq i\leq m$ and denote by $\bar u(r)$, $\bar v(r)$ (resp $\bar u_i(r)$, $\bar v_i(r)$) the spherical average of  $u$ and $v$  (resp $u_i$ and $v_i$) on the sphere $\partial B_r(x_0)$, that is,
$$
\bar u(r)=\fint_{\partial B_r(x_0)} u(y)d\sigma(y)\quad\mbox{ and }\quad \bar u_i(r)=\fint_{\partial B_r(x_0)} u_i(y)d\sigma(y),
$$
$$
\bar v(r)=\fint_{\partial B_r(x_0)} v(y)d\sigma(y)\quad\mbox{ and }\quad \bar v_i(r)=\fint_{\partial B_r(x_0)} v_i(y)d\sigma(y).
$$
Then
\begin{equation}\label{system}
\begin{cases}
-\Delta \bar u=\bar u_1\, ,\, \ \qquad\ \  -\Delta \bar v=\bar v_1\,,\\
-\Delta \bar u_1=\bar u_2\, , \qquad\ \  -\Delta \bar v_1=\bar v_2\,,\\
\cdots\cdots\cdots\cdots\\
-\Delta \bar u_{m-2}=\bar u_{m-1}\, , -\Delta \bar v_{m-2}=\bar v_{m-1}\,,\\
-\Delta \bar u_{m-1} \ge \di \fint_{\partial B_r(x_0)} \big(\Phi(|x|)*v^{p_1}\big)u^{q_1}  d\sigma\geq 0\, ,\\[0.15in]
-\Delta \bar v_{m-1} \geq \di \fint_{\partial B_r(x_0)} \big(\Psi(|x|)*u^{p_2}\big)v^{q_2}  d\sigma\geq 0. 
\end{cases}
\end{equation}
From \eqref{system} one has $-\Delta \bar u_{m-1} \ge 0$ which yields $-r^{1-N}\Big(r^{N-1}\bar u_{m-1}'\Big)'\geq 0$ for all $r>0$ and 
$$
\bar u_{m-1}(0)=(-\Delta)^{m-1}u(x_0)<0.
$$ 
By integration one gets
$$
\bar u'_{m-1}(r)\leq 0\quad \mbox{ and }\quad \bar u_{m-1}(r)\leq \bar u_{m-1}(0)=u_{m-1}(x_0)<0,
$$
for all $r\geq 0$.
We can rewrite the last estimate as
\begin{equation}\label{eqzero}
(-\Delta)^{m-1}  \bar u (r)\leq  (-\Delta)^{m-1} \bar u(0)<0 \quad\mbox{ for all }r\geq 0.
\end{equation}

\noindent{\bf Case 1: } $m$ is odd. From  \eqref{eqzero} one has
$$
\Delta^{m-1} \bar u(r)\leq \Delta^{m-1}\bar u(0)<0\quad\mbox{ for all }r\geq 0.
$$
Integrating twice the above inequality we obtain
$$
\Delta^{m-2} \bar u (r)\leq \Delta^{m-2} \bar u(0)+\frac{\Delta^{m-1}u(0)r^2}{2N} \quad\mbox{ for all }r\geq 0
$$
and proceeding further we deduce
\begin{equation}\label{eq1}
\bar u(r)\leq \bar u(0)+\sum_{k=1}^{m-1}\frac{\Delta^{k}\bar u(0)}{\di \Pi_{j=1}^k[(2j)(N+2j-2)]}r^{2k}.
\end{equation}
Since $\Delta^{m-1}\bar u(0)=(-\Delta )^{m-1}u(x_0)<0$, we deduce from \eqref{eq1} that 
$$
\bar u(r)\to -\infty\quad\mbox{ as } r\to \infty,
$$ 
which contradicts the fact that $u\geq 0$.

\noindent{\bf Case 2:} $m$ is even. Hence $m\geq 2$. From \eqref{eqzero} we find
$$
\Delta^{m-1}  \bar u (r)\geq  \Delta^{m-1} \bar u(0)>0 \quad\mbox{ for all }r\geq 0.
$$
In the same manner as we derived \eqref{eq1} it follows that  for any $1\leq i\leq m$ one has
$$
\Delta^{m-i} \bar u(r)\geq \Delta^{m-i} \bar  u(0)+\sum_{k=1}^{i-1}\frac{\Delta^{m-i+k}\bar u(0)}{\di \Pi_{j=1}^k[(2j)(N+2j-2)]}r^{2k}.
$$
In particular, for $i=m$ we find
$$
\bar u(r)\geq \bar u(0)+\sum_{k=1}^{m-1}\frac{\Delta^{k}\bar u(0)}{\di \Pi_{j=1}^k[(2j)(N+2j-2)]}r^{2k}.
$$
Since $\Delta^{m-1}\bar u(0)>0$ and $m\geq 2$ it follows from the above estimate that 

\begin{equation}\label{S3-E2}
\bu(r)\ge C_1 r^{2(m-1)}-C_2\quad\mbox{ for all }r\geq 0,
\end{equation}
for some constants $C_1, C_2>0$.

\noindent{\bf Case 2a:} Assume that \eqref{ss1} holds. 
Let $\varphi$ be as in Lemma \ref{lg}.
In the same way as in the proof of \eqref{eqg1} we have
\begin{equation}\label{eqg3}
\begin{cases}
\displaystyle \int_{\R^N} u\varphi dx\geq CR^{-N+2m}\Big(\int_{\R^N} v^{\frac{\tau}{2}}\varphi   dx\Big)^2,\\[0.2in]
\displaystyle \int_{\R^N} v\varphi dx\geq CR^{-N+2m}\Big(\int_{\R^N} u^{\frac{\theta}{2} }\varphi   dx\Big)^2,\\
\end{cases}
\end{equation}
where
$$
\tau=p_1+q_1\geq 2\quad\mbox{ and }\quad  \theta=p_2+q_2\geq 2.
$$
By H\"older's inequality we find
\begin{equation}\label{eqg4}
\begin{cases}
\displaystyle 
\Big(\int_{\R^N} v^{\frac{\tau}{2}}\varphi   dx\Big)^2\geq C R^{-N(\tau -2)} \Big(\int_{\R^N} v\varphi   dx\Big)^{\tau},\\[0.2in]
\displaystyle 
\Big(\int_{\R^N} u^{\frac{\theta}{2}}\varphi   dx\Big)^2\geq C R^{-N(\theta -2)} \Big(\int_{\R^N} u\varphi   dx\Big)^{\theta}.
\end{cases}
\end{equation}
From \eqref{eqg3} and \eqref{eqg4} we deduce
\begin{equation}\label{eqg5}
\begin{cases}
\displaystyle \int_{\R^N} u\varphi dx\geq C R^{-N+2m-N(\tau -2)} \Big(\int_{\R^N} v\varphi   dx\Big)^{\tau},\\[0.2in]
\displaystyle \int_{\R^N} v\varphi dx\geq CR^{-N+2m-N(\theta -2)} \Big(\int_{\R^N} u\varphi   dx\Big)^{\theta}.\\
\end{cases}
\end{equation}
We use the second estimate of \eqref{eqg5} in the first one to obtain
$$
\int_{\R^N} u\varphi dx\geq C R^{-N+2m-N(\tau -2)-(N-2m)\tau-N(\theta -2)\tau } 
\Big(\int_{\R^N} u\varphi   dx\Big)^{\tau \theta},
$$
which we arrange as
$$
CR^{N-\frac{2m(1+\tau)}{\tau \theta-1}}\geq \int_{B_R} u dx.
$$
Using \eqref{S3-E2} we find
$$
\begin{aligned}
CR^{N-\frac{2m(1+\tau)}{\tau \theta-1}}&\geq 
\int_{B_R}udx
=\sigma_{N}\int_0^Rr^{N-1}\bu dr\\
& \ge C_3 R^{N+2(m-1)} -C_4R^N \quad\mbox{ for $R>1$ large},
\end{aligned}
$$
where $C_3, C_4>0$ and $\sigma_N$ denotes the surface area of the unit sphere in $\R^N$. Comparing the exponents of $R$ in the above inequality we raise a contradiction, since $C_3>0$.  This finishes the proof of Step 1 in this case.

\noindent{\bf Case 2b:} Assume that \eqref{ss2} holds. From \eqref{S3-E2} one can find $r_0>0$ and a constant $c>0$ such that
\begin{equation}\label{cont1}
\bu(r)\ge cr^{2(m-1)}\quad\mbox{ for all }r\geq r_0.
\end{equation}
Using the fact that $r^N\Psi(r)\to \infty$ as $r\to \infty$, by taking $r_0>1$ large enough we may also assume that
\begin{equation}\label{cont2}
r^N\Psi\Big(\frac{r}{2}\Big)\geq 1 \quad\mbox{ for all }r\geq r_0.
\end{equation}
To raise a contradiction, we next return to condition \eqref{eqkk} for $u$ (in which we replace $p$ with $p_2$). From \eqref{cont1}-\eqref{cont2}, co-area formula and Jensen's inequality we obtain:

$$
\begin{aligned}
\infty>\int_{|y|>r_0}u^{p_2}(y)\Psi\Big(\frac{|y|}{2}\Big) dy & =\int_{r_0}^\infty \int_{|y|=r} u^{p_2}(y)\Psi\Big(\frac{|y|}{2}\Big) d\sigma(y)\, dr\\[0.1in]
&=\int_{r_0}^\infty \Psi\Big(\frac{r}{2}\Big) \int_{|y|=r} u^{p_2}(y) d\sigma(y)\, dr\\[0.1in]
&\geq \sigma_N\int_{r_0}^\infty r^{N-1} \Psi\Big(\frac{|r|}{2}\Big)\bu^{p_2}(r)   dr\\[0.1in]
&\geq C\int_{r_0}^\infty r^{2p_2(m-1)-1} r^{N} \Psi\Big(\frac{|r|}{2}\Big)  dr\\
&\geq C\int_{r_0}^\infty r^{2p_2(m-1)-1}   dr=\infty, 
\end{aligned}
$$
which is a  contradiction and concludes  the proof in Step 1.

\medskip

\noindent{\bf Step 2: }
We have $(-\Delta)^{m-i}u\geq 0$ and $(-\Delta)^{m-i}v\geq 0$ in $\R^N$ for any $1\leq i\leq m$.

From Step 1 we know that $(-\Delta)^{m-1}u\geq 0$, $(-\Delta)^{m-1}v\geq 0$ in $\R^N$. Letting $u_{m-2}=(-\Delta)^{m-2}u$ and $v_{m-2}=(-\Delta)^{m-2}v$ we want to show next that $u_{m-2}\geq 0$ and $v_{m-2}\geq 0$ in $\R^N$. Suppose to the contrary that there exists $x_0\in \R^N$ so that $u_{m-2}(x_0)<0$. We next take the spherical average with respect to spheres centred at $x_0$ and proceed as in Step 1 by discussing separately the cases $m$ is odd and $m$ is even in order to raise a contradiction. Thus, $(-\Delta)^{m-2}u\geq 0$, $(-\Delta)^{m-2}v\geq 0$ in $\R^N$. We proceed further until we get $-\Delta u\geq 0$, $-\Delta v \geq 0$ in $\R^N$.  
\end{proof}

\begin{proof}[Proof of Theorem~\ref{thmain1}]
Let $\Psi=\Phi$ and $(p_1,q_1)=(p_2,q_2)=(p,q)$.
Suppose $u$ is a nonnegative solution of (\ref{main}).
If $u\equiv 0$, then the conclusion clearly holds.
If $u\not\equiv 0$, then $(u,u)$ is a non-negative solution of (\ref{nlc}). 
By Theorem~\ref{thmainext} we see that, for all $1\le i\le m$, $(-\Delta)^iu\ge 0$ in $\R^N$.
\end{proof}

\section{Proof of Theorem \ref{thmain2}}

(i) By Theorem \ref{thmain1} we see that, for $1\le j\le m$, $(-\Delta)^ju\ge 0$ in $\R^N$.
In particular $-\Delta u\ge 0$ in $\R^N$.
Since $N=1,2$, it is well known that a nonnegative superharmonic function is constant.
Thus, $u=c$ in $\R^N$.
By (\ref{main}) it follows that $(\Psi(|x|)*u^p)u^q=0$ in $\R^N$.
This clearly yields $u=0$ in $\R^N$, otherwise there would exist $x_0\in\R^N$ such that $u(x_0)>0$ and hence $(\Psi(|x|)*u^p)u^q>0$ at $x_0$.

(ii1) By estimate \eqref{inqq9} in Proposition \ref{lemv}
we find $u\geq c|x|^{2m-N}$ in $\R^n\setminus B_1$, for some $c>0$.  Thus, by the hypothesis (ii1) in Theorem \ref{thmain2} we find
$$
\int_{|y|>1}
u^p(y)\Psi\Big(\frac{|y|}{2}\Big)dy\geq c\int_{|y|>1}|y|^{-p(N-2m)}\Psi(|y|)dy=\infty,
$$
which contradicts \eqref{eqkk}. 

The proof of part (ii2) follows from the proof of Theorem \ref{S4T2}.
\qed

\section{Proof of Theorem \ref{thmain3}} 

Assume first that \eqref{eqpqm1} has a non-negative solution $u\not\equiv 0$. Then, by estimate \eqref{inqq9} in Proposition \ref{lemv} we have $u\geq c|x|^{2m-N}$ in $\R^N\setminus B_1$, where $c>0$ is a constant. It is easy to check that the condition $p>(N-\alpha)/(N-2m)$ and \eqref{eqpqm2}$_2$ follow from Theorem \ref{thmain2} with $\Psi(r)=r^{-\alpha}$. 

It remains to prove that $q>(N-\alpha)/(N-2m)$. If $\alpha\geq 2m$ then this is clearly true,  since $q>1$.   Assume next that $\alpha<2m$. 

For $x\in \R^N\setminus B_1$ and $1<|y|<2$ we have $|x-y|\leq 3|x|$. Thus, 
$$
\begin{aligned}
|x|^{-\alpha}*u^p&\geq \int_{\R^N} \frac{f(y)}{|x-y|^{\alpha}} dy\geq \int_{1< |y|< 2} \frac{u^p(y)}{|x-y|^{\alpha}} dy\\
&\geq  \int_{1< |y|< 2} \frac{u^p(y)dy}{(3|x|)^{\alpha}}\geq C |x|^{-\alpha}.
\end{aligned}
$$
Thus, $|x|^{-\alpha}*u^p\geq C|x|^{-\alpha}$ for all $x\in \R^N\setminus B_1$ and $u$ satisfies
$$
(-\Delta)^m u\geq c|x|^{-\alpha}u^q\quad \mbox{ in } \R^N\setminus B_1,
$$
for some $c>0$. We know (see, e.g., \cite[Example 5.2]{MP2001}) that the above inequality has no solutions $u\geq 0$, $u\not\equiv 0$,  if $\alpha<2m$ and $1<q\leq (N-\alpha)/(N-2m)$.  Hence $q>(N-\alpha)/(N-2m)$.

Assume now that \eqref{eqpqm2} holds and let us construct a positive solution to \eqref{eqpqm1}. First, we write \eqref{eqpqm2} in the form
$$
(N-2m)(p+q-1)>N-\alpha+2m\quad (N-2m)p>N-\alpha \quad \mbox{ and }\quad (N-2m)(q-1)>2m-\alpha.
$$
Thus, we can choose $\kappa\in (0, N-2m)$ such that 
\begin{equation}\label{eqka}
\begin{cases}
\kappa (p+q-1)>N-\alpha+2m,\\
\kappa p>N-\alpha,\\
\kappa (q-1)>2m-\alpha,\\
p \kappa\neq N. 
\end{cases}
\end{equation}
For $a\geq 0$ we define 
$$
F(a,x)=(-\Delta)^m \Big\{(a+|x|^2)^{-\kappa/2}  \Big\}\quad\mbox{ for all }x\in \R^N\setminus\{0\}.
$$
Then,
$$
F(a,x)=(a+|x|^2)^{-\frac{\kappa}{2}-2m}\sum_{j=0}^m b_j(a)|x|^{2j}
\quad\mbox{ for all }x\in \R^N,
$$
where $b_j(a)\in \R$. In particular, for $a=0$ we find
\begin{equation}\label{not1}
F(0,x)=|x|^{-\kappa-4m}\sum_{j=0}^m b_j(0)|x|^{2j}
\quad\mbox{ for all }x\in \R^N\setminus\{0\}.
\end{equation}
On the other hand, by direct computation one has
\begin{equation}\label{not2}
F(0,x)=(-\Delta)^m \Big\{|x|^{-\kappa}  \Big\}=\prod_{j=1}^m\Big[(\kappa+2j-2)(N-\kappa-2j)\Big]|x|^{-\kappa-2m}>0,
\end{equation}
since $0<\kappa<N-2m$. Comparing \eqref{not1} and \eqref{not2} we find $b_{2m}(0)>0$. By the continuous dependence on the data, we can find now $a>0$ such that $b_{2m}(a)>0$. Also,
$$
\lim_{|x|\to \infty}\frac{F(a,x)}{|x|^{-k-2m}}=b_{2m}(a)>0.
$$
Thus, there exist $c>0$ and $R>1$ such that $F(a,x)\geq c|x|^{-\kappa-2m}$ for $x\in \R^N\setminus B_R$.

Let now $v(x)=(a+|x|^2)^{-\kappa/2}$, where $a>0$ satisfies $b_{2m}(a)>0$. By the above estimates we have
\begin{equation}\label{eqe1}
(-\Delta)^m v\geq c|x|^{-\kappa-2m}\quad\mbox{ in } \; \R^N\setminus B_R.
\end{equation}
Let $\varphi\in C^1_c(\R^N)$, $0\leq \varphi\leq 1$ such that supp $\varphi\subset B_{2R}$ and $\varphi\equiv 1$ on $B_R$.
For $M>1$ define
$$
V(x)=v(x)+M \gamma_0\int_{\R^N}\frac{\varphi(y)}{|x-y|^{N-2m}}dy \quad\mbox{ for all }x\in \R^N,
$$
where $\gamma_0>0$ is  a normalizing constant such that
$$
(-\Delta)^m\big(\gamma_0 |x|^{2m-N}\big)=\delta_0\quad\mbox{ in }{\mathcal D}'(\R^N),
$$
and $\delta_0$ denotes the Dirac mass concentrated at the origin.

Thus, $V\in C^{2m}(\R^N)$, $V>0$ in $\R^N$ and from  \eqref{eqe1} we have
\begin{equation}\label{eqe2}
(-\Delta)^m V\geq c|x|^{-\kappa-2m}\quad\mbox{ in } \; \R^N\setminus B_{R}.
\end{equation}
Also, by taking $M>1$ large enough we have
\begin{equation}\label{eqe3}
(-\Delta)^m V=(-\Delta)^mv+M>0 \quad\mbox{ in } \; \overline B_R.
\end{equation}
Observe that for $x\in \R^N\setminus B_{4R}$ and $y\in B_{2R}$ we have $|x-y|\geq |x|-|y|\geq |x|/2$. Thus,
$$
\begin{aligned}
\int_{\R^N}\frac{\varphi(y)}{|x-y|^{N-2m}}dy&=\int_{B_{2R}} \frac{\varphi(y)}{|x-y|^{N-2m}}dy\\
&\leq 2^{N-2m}|x|^{2m-N}\int_{B_{2R}}\varphi(y) dy\\
&\leq C|x|^{2m-N}.
\end{aligned}
$$
Using this estimate in the definition of $V$ together with $0<\kappa<N-2m$ it follows that
\begin{equation}\label{eqe4}
V(x)\leq C_0|x|^{-\kappa} \quad\mbox{ for all }x\in \R^N\setminus B_{R/2},
\end{equation}
for some constant $C_0>0$.

We next evaluate the convolution term $(|x|^{-\alpha}*V^p)V^q$ and indicate how to construct a positive solution to \eqref{eqpqm1}. Using \eqref{eqe4} we can apply Lemma \ref{lbas} for $f=V^p$, $\beta=\kappa p>N-\alpha$ and $\rho=R/2$. It follows that for any $x\in \R^N\setminus B_R$ we have

$$
(|x|^{-\alpha}*V^p)V^q\leq c|x|^{-\kappa q}\int_{\R^N}\frac{V(y)^pdy}{|x-y|^{\alpha}}
\leq C\left\{
\begin{aligned}
&|x|^{N-\alpha-\kappa (p+q)}&&\quad\mbox{ if }\kappa p<N,\\
&|x|^{-\alpha-\kappa q} &&\quad\mbox{ if }\kappa p>N.
\end{aligned}
\right.
$$ 
Using this last estimate together with \eqref{eqe2} and \eqref{eqka}$_1$, \eqref{eqka}$_3$  we deduce
\begin{equation}\label{eeee1}
(-\Delta)^m V\geq C_1 (|x|^{-\alpha}*V^p)V^q \quad\mbox{ in } \; \R^N\setminus B_R,
\end{equation}
for some $C_1>0$. 
Since $(-\Delta)^m V$ and $(|x|^{-\alpha}*V^p)V^q $ are continuous and positive functions on the compact $\overline B_R$ (see \eqref{eqe3}), one can find $C_2>0$ such that 
\begin{equation}\label{eeee2}
(-\Delta)^m V\geq C_2 (|x|^{-\alpha}*V^p)V^q \quad\mbox{ in } \;  B_R.
\end{equation}
Thus, letting $C=\min\{C_1,C_2\}>0$ and $U=C^{1/(p+q-1)}V$, it follows that $U\in C^{2m}(\R^N)$ is positive and that from \eqref{eeee1}-\eqref{eeee2} one has 
$$
(-\Delta)^m U\geq  (|x|^{-\alpha}*U^p)U^q \quad\mbox{ in } \; \R^N,
$$
which concludes our proof.
\qed

\section{Proof of Theorem \ref{S4T2} and Theorem \ref{S4T3} }


\noindent {\bf Proof of Theorem \ref{S4T2}.} Let $L>0$ denote the positive limit in \eqref{S4T2E0}. Thus, one can find an increasing sequence $\{R_i\}\subset (0, \infty)$ that tends to infinity and such that for all $i\geq 1$ one has
\begin{equation}\label{ri1}
\min\Big\{ R_i^{2N-(N-2m)(p_{k\ell}+q_{k\ell})}\Psi_{k\ell}(R_i)\,,\, 
r^{2N-(N-2m)(p_{\ell k}+q_{\ell k})}\Psi_{\ell k}(R_i)\Big\}>\frac{L}{2}.
\end{equation}
By Theorem \ref{thmainext} we deduce that $u_k$ and $u_\ell$ are poly-superharmonic. Further, by estimate \eqref{inqq9} in Proposition \ref{lemv} there exists $c>0$ such that 
\begin{equation}\label{esst}
u_k(x), u_\ell(x)\geq c|x|^{2m-N}\quad\mbox{ in }\R^N\setminus B_1.
\end{equation}

Let $\phi$ be the positive eigenfunction of $-\Delta$ in the unit ball $\overline B_1$ corresponding to the eigenvalue $\lambda_1>0$.
We normalize $\phi$ such that $0\le\phi\le 1$ in $B_1$ and $\max_{\overline B_1}\phi(x)=1$.
Let $\varphi_{i}(x)=\phi(x/R_i)$.
Multiplying by $\varphi_{i}$ in the inequality of \eqref{mainsystem} that corresponds to $u_k$ we find
\begin{align*}
\int_{B_{R_i}} \left(\Psi_{k\ell}(|x|)*u_{\ell}^{p_{k\ell}}\right) u_\ell^{q_{k\ell }}\varphi_i
&\le\int_{B_{R_i}}(-\Delta)^mu_k\varphi_i\\
&=\int_{B_{R_i}}(-\Delta)^{m-1}u_k(-\Delta)\varphi_i
+\int_{\partial B_{R_i}}(-\Delta)^{m-1}u_k\frac{\partial\varphi_i}{\partial n}\\
&\le\frac{\lambda_1}{R_i^2}\int_{B_{R_i}}(-\Delta)^{m-1}u_k\varphi_i,
\end{align*}
where we used $(-\Delta)^{m-1}u_k\ge 0$ by Theorem~\ref{thmainext} and that, by Hopf lemma, $\partial\varphi_i/\partial n<0$ on $\partial B_{R_i}$.
Proceeding further one finds
\begin{equation}\label{S4L1E2}
\int_{B_{R_i}} \left(\Psi_{k\ell}(|x|)*u_{\ell}^{p_{k\ell}}\right) u_\ell^{q_{k\ell }}\varphi_i
\le\left(\frac{\lambda_1}{R_i^2}\right)^m
\int_{B_{R_i}}u_k\varphi_i.
\end{equation}
Let us next estimate the integral in the left-hand side of (\ref{S4L1E2}).
If $x\in B_{R_i}$, then  one  has
\[
\Psi_{k\ell}(|x|)*u_{\ell}^{p_{k\ell}} \ge
\int_{B_{R_i}}\Psi_{k\ell}(|x-y|)u_\ell^{p_{k\ell}}(y)dy
\ge \Psi_{k\ell}(2R_i)\int_{B_{R_i}}u_\ell^{p_{k\ell}}(y)dy.
\]
Thus, by the fact that $0\le\varphi_i\le 1$ one has
\begin{equation}\label{S4L1E3}
\int_{B_{R_i}} \left(\Psi_{k\ell}(|x|)*u_{\ell}^{p_{k\ell}}\right) u_\ell^{q_{k\ell }}\varphi_i
\ge \Psi_{k\ell}(2R_i)
\left(\int_{B_{R_i}}u_\ell^{p_{k\ell}}\varphi_i\right)
\left(\int_{B_{R_i}}u_\ell^{q_{k\ell}}\varphi_i\right).
\end{equation}
Combining (\ref{S4L1E3}) with (\ref{S4L1E2}) we obtain
\begin{equation}\label{S4L1E4}
\Psi_{k\ell}(2R_i)\left(\int_{B_{R_i}}u_{\ell}^{p_{k\ell}}\varphi_i\right)\left(\int_{B_{R_i}}u_{\ell}^{q_{k\ell}}\varphi_i\right)
\le\left(\frac{\lambda_1}{R_i^2}\right)^m\int_{B_{R_i}}u_k\varphi_i.
\end{equation}

Let $\tau=p_{k\ell}+q_{k\ell}\ge 2$.
By H\"{o}lder's inequality we have
\begin{equation}\label{S4L1E5}
\left(\int_{B_{R_i}}u_{\ell}^{\frac{\tau}{2}} \varphi_i\right)^2
\le
\left(\int_{B_{R_i}}u_{\ell}^{p_{k\ell}}\varphi_i\right)
\left(\int_{B_{R_i}}u_{\ell}^{q_{k\ell}}\varphi_i\right).
\end{equation}
Again by H\"{o}lder's inequality we derive
\[
\int_{B_{R_i}}u_{\ell}\varphi_i
\le
\left(\int_{B_{R_i}}u_{\ell}^{\frac{\tau}{2}}\varphi_i\right)^{\frac{2}{\tau}}
\left(\int_{B_{R_i}}\varphi_i\right)^{1-\frac{2}{\tau}}
\le
CR_i^{N\left(1-\frac{2}{\tau}\right)}
\left(\int_{B_{R_i}}u_{\ell}^{\frac{\tau}{2}}\varphi_i\right)^{\frac{2}{\tau}},
\]
and hence
\begin{equation}\label{S4L1E5+}
\left(\int_{B_{R_i}}u_{\ell}\varphi_i\right)^{\tau}
\le
CR_i^{N(\tau-2)}
\left(\int_{B_{R_i}}u_{\ell}^{\frac{\tau}{2}}\varphi_i\right)^2.
\end{equation}
By (\ref{S4L1E4}), (\ref{S4L1E5}) and (\ref{S4L1E5+}) we have
\[
R_i^{2m-N(\tau-2)}\Psi_{k\ell}(2R_i)
\left(\int_{B_{R_i}}u_{\ell}\varphi_i\right)^{\tau}
\le
C\int_{B_{R_i}}u_k\varphi_i,
\]
which we write it as
\[
R_i^{2N-(N-2m)\tau}\Psi_{k\ell}(2R_i)
\left(\int_{B_{R_i}}u_{\ell}\varphi_i\right)^{\tau}
\le
CR_i^{2m(\tau-1)}
\int_{B_{R_i}}u_k\varphi_i
\]

Using \eqref{ri1} it follows that for $i\geq 1$ large enough we have
\begin{equation}\label{S4L1E6a}
\frac{L}{2}
\left(\int_{B_{R_i}}u_{\ell}\varphi_i\right)^{p_{k\ell}+q_{k\ell}}
\le CR_i^{2m(p_{k\ell}+q_{k\ell}-1)}\int_{B_{R_i}}u_k\varphi_i.
\end{equation}
Similarly,  we have
\begin{equation}\label{S4L1E6b}
\frac{L}{2}
\left(\int_{B_{R_i}}u_{k}\varphi_i\right)^{p_{\ell k}+q_{\ell k}}
\le CR_i^{2m(p_{\ell k}+q_{\ell k}-1)}\int_{B_{R_i}}u_\ell\varphi_i.
\end{equation}
Multiplying (\ref{S4L1E6a}) with (\ref{S4L1E6b}) and using the fact that 
\[
\left( \int_{B_{R_i}}u_{k}\varphi_i \right)
\left( \int_{B_{R_i}}u_{\ell}\varphi_i \right)>0
\ \ \textrm{for large}\ i,
\]
we have
$$
\left( \int_{B_{R_i}}u_{\ell}\varphi_i \right)^{p_{k\ell}+q_{k\ell}-1}
\left( \int_{B_{R_i}}u_{k}\varphi_i \right)^{p_{\ell k}+q_{\ell k}-1}
\le C R_i^{2m(p_{k\ell}+q_{k\ell}-1)}R_i^{2m(p_{\ell k}+q_{\ell k}-1)}.
$$
From here we deduce that there exists a subsequence $\{R_i\}$ (still denoted in the following by $\{R_i\}$) such that\footnote{We make use of the following basic argument: if $a_ib_i\leq x_iy_i$ then either $a_i\leq x_i$ or $b_i\leq y_i$ (we argue by contradiction to prove this fact). Since $i\geq 1$ can be any (large) positive integer, along a subsequence we have either $a_i\leq x_i$ or $b_i\leq y_i$.}

\begin{itemize}
\item  either $\displaystyle \left( \int_{B_{R_i}}u_{\ell}\varphi_i \right)^{p_{k\ell}+q_{k\ell}-1}\leq C R_i^{2m(p_{k\ell}+q_{k\ell}-1)}$;
\item or 
$
\displaystyle 
\left( \int_{B_{R_i}}u_{k}\varphi_i \right)^{p_{\ell k}+q_{\ell k}-1}
\le C R_i^{2m(p_{\ell k}+q_{\ell k}-1)}$.
\end{itemize}
Assume the second assertion holds. This yields
$$
\int_{B_{R_i}}u_{k}\varphi_i 
\le C R_i^{2m}.
$$
Using this last estimate in \eqref{S4L1E2} we deduce 
$\left(\Psi_{k\ell}(|x|)*u_{\ell}^{p_{k\ell}}\right)u_{\ell}^{q_{k\ell}}\in L^1(\RN)$, and so
\begin{equation}\label{S4L1E9}
\int_{B_{R_i}\setminus B_{R_i/2}}\left(\Psi_{k\ell}(|x|)*u_{\ell}^{p_{k\ell}}\right)u_{\ell}^{q_{k\ell}}\varphi_i
\to 0\ \ \textrm{as}\ \ i\to\infty.
\end{equation}
We may estimate the above integral as we did in (\ref{S4L1E3}) to obtain
\[
\int_{B_{R_i}\setminus B_{R_i/2}}\left(\Psi_{k\ell}(|x|)*u_{\ell}^{p_{k\ell}}\right)u_{\ell}^{q_{k\ell}}\varphi_i
\ge \Psi_{k\ell}(2R_i)
\left(\int_{B_{R_i}\setminus B_{R_i/2}}u_{\ell}^{p_{k\ell}}\varphi_i\right)
\left(\int_{B_{R_i}\setminus B_{R_i/2}}u_{\ell}^{q_{k\ell}}\varphi_i\right).
\]
Finally, we use \eqref{esst} in the above inequality to deduce
\[
\int_{B_{R_i}\setminus B_{R_i/2}}\left(\Psi_{k\ell}(|x|)*u_{\ell}^{p_{k\ell}}\right)u_{\ell}^{q_{k\ell}}
\ge CR_i^{2N-(N-2m)(p_{k\ell}+q_{k\ell})}\Psi_{k\ell}(2R_i)>CL>0,
\]
for large $i$, which contradicts (\ref{S4L1E9}) and concludes our proof.
\qed

\medskip

The proofs of Corollaries \ref{corol1} and  \ref{corol4}~(i) follow immediately. 

\bigskip

\noindent{\bf Proof of Theorem \ref{S4T3}.} The proof of Theorem \ref{S4T3} can be carried out in the same way as above. The only difference is that we cannot apply Theorem \ref{thmainext} to derive that $u_k$ and $u_\ell$ satisfy \eqref{esst}. Instead, 
we apply Lemma \ref{lemv} to deduce that $u_k$ and $u_\ell$ are poly-superharmonic. Further, by the estimate 
\eqref{inqq9} one has that \eqref{esst} holds. From now on, we follow the above proof line by line. \qed

\bigskip

\noindent{\bf Proof of Corollary \ref{corol4}~(ii).} Let $(u_1,\ldots,u_n)$ be a nontrivial nonnegative solution of (\ref{AT1E0}).
Assume, without loss of generality, that $u_1\not\equiv 0$.
Let $i\in\{2,\ldots,n\}$.
Since $e_{1i}=e_{i1}=1$, by Theorem~\ref{S4T3} we see that $u_i\equiv 0$.
This indicates that $u_i\equiv 0$ for each $i\in\{2,\ldots,n\}$.
The conclusion holds. \qed

\bigskip

\noindent{\bf Acknowledgment.} The authors would like to thank the anonymus referee for the careful reading of our manuscript and for the suggestions which led to an improvement of our presentation.

\end{document}